\documentclass[11pt]{amsart}
\usepackage{amsfonts}
\usepackage{amsmath}
\usepackage{amssymb,latexsym}
\usepackage[mathcal]{eucal}
\usepackage{amscd}

\input xy
\xyoption{all}

\newcommand{\Fcal}{\mathcal{F}}

\newcommand{\Ical}{\mathcal{I}}
\newcommand{\Jcal}{\mathcal{J}}

\newcommand{\Lcal}{\mathcal{L}}

\newcommand{\Pcal}{\mathcal{P}}

\newcommand{\ch}{\mathbf{1}}

\newcommand{\al}{\alpha}
\newcommand{\Ga}{\Gamma}
\newcommand{\ga}{\gamma}
\newcommand{\del}{\delta}

\newcommand{\br}{\vspace{3 mm}}

\newcommand{\cls}{{\rm{cls\,}}}

\swapnumbers
\theoremstyle{plain}
\newtheorem{thm}{Theorem}[section]
\newtheorem{cor}[thm]{Corollary}

\theoremstyle{definition}

\newtheorem{defns}[thm]{Definitions}

\newtheorem{prob}[thm]{Problem}

\numberwithin{thm}{section}

\begin{document}

\title
{Translation-finite sets}

\author[]{Eli Glasner}
\address{Department of Mathematics,
Tel-Aviv University, Ramat Aviv, Israel}
\email{glasner@math.tau.ac.il}
\urladdr{http://www.math.tau.ac.il/$^\sim$glasner}


\date{June 9, 2011}

\begin{abstract}
The families of right (left) translation finite subsets of a discrete infinite group 
$\Ga$ are defined and shown to be ideals. Their kernels $Z_R$ and $Z_L$
are identified as the closure of the set of products $pq$ ($p\cdot q$) in the 
\v{C}ech-Stone compactification $\beta\Ga$. Consequently it is shown 
that the map $\pi: \beta\Ga \to \Ga^{WAP}$,
the canonical semigroup homomorphism from $\beta\Ga$ onto $\Ga^{WAP}$,
the universal semitopological semigroup compactification of $\Ga$,
is a homeomorphism on the complement of $Z_R \cup Z_L$.
\end{abstract}



\maketitle

\section*{Introduction}
This note is an elaboration on the beautiful work of Ruppert \cite{Ru} from 1985.
Given a discrete infinite group $\Ga$ we define right and left versions of the 
combinatorial property (of subsets of $\Ga$) of being translation finite.
Then, using the ultrafilter representation of the \v{C}ech-Stone compactification 
$\beta\Ga$, we show that the collections of sets with these properties form ideals
(Theorem \ref{Fcal}).
This yields a new proof of Ruppert's theorem which asserts that the collection
of translation finite sets forms an ideal. We then use these results to obtain 
some unexpected information about the map $\pi: \beta\Ga \to \Ga^{WAP}$,
the canonical semigroup homomorphism from $\beta\Ga$ onto $\Ga^{WAP}$,
the universal semitopological semigroup compactification of $\Ga$
(Theorem \ref{pi}).

\section{The $C^*$-algebras $\ell_\infty(\Ga)$ and $WAP(\Ga)$}
Let $\Ga$ be a countable discrete infinite group with unit element $e$. 
We briefly review some basic properties of the $C^*$-algebras
$\ell_\infty(\Ga)$, of bounded complex-valued functions on $\Ga$, and  
$WAP(\Ga)$, the closed subalgebra comprising the weakly almost periodic functions  on $\Ga$.  
Recall that $f \in \ell_\infty(\Ga)$ is {\em weakly almost periodic} if its orbit under translations $\{f \circ \ga: \ga \in \Ga\}$ is a weakly precompact 
subset of the Banach space $\ell_\infty(\Ga)$.
We are mostly interested in their Gelfand (or maximal ideal) spaces: 
$\beta\Ga$, the \v{C}ech-Stone compactification of $\Ga$, and $\Ga^{WAP}$, 
the universal WAP-compactification of $\Ga$, respectively.

The compactification $\beta\Ga$ can be viewed as the collection of 
ultrafilters on $\Ga$,
where an element $\ga \in \Ga$ is presented as the principal ultrafilter
$e_\ga = \{A \subset \Ga : \ga \in A\}$. Then the left translation of 
an ultrafilter $q \in \beta\Ga$ by $\ga$ is the ultrafilter
$\ga q = \{ A \subset \Ga : \ga^{-1} A \in q\}$
(note that this extends the product on $\Ga$ as $\ga e_\del = e_{\ga\del}$) . 
These translations define a left action of $\Ga$ on $\beta\Ga$ and the resulting pointed dynamical system 
$(\beta\Ga,e,\Ga)$ is the {\em universal ambit} (or point transitive pointed system).
That is, for any point transitive pointed $\Ga$ dynamical system $(Y,y_0,\Ga)$
there is a unique homomorphism of pointed dynamical systems
$\pi : (\beta\Ga,e,\Ga) \to (Y,y_0,\Ga)$.

This $\Ga$ action on $\beta\Ga$ can be extended to 
a multiplication on $\beta\Ga$ as follows: for $p,q \in \beta\Ga$
$$
m_R(p,q)= pq = \{A \subset \Ga : \{\al \in \Ga : \al^{-1} A \in q\} \in p\}.
$$
This multiplication has the property that for each fixed $q \in \beta\Ga$
the map $R_q : \beta\Ga \to \beta\Ga$, defined by $p \mapsto
pq = m_R(p,p)$ is continuous. Thus this product makes $\beta\Ga$
a {\em right topological semigroup}. It can be shown that this 
right topological semigroup can be identified with the enveloping
semigroup $E(\beta\Ga,\Ga)$ of the dynamical system $(\beta\Ga,\Ga)$.

One can also define a left product on $\beta\Ga$ by
$$
m_L(p,q)= p\cdot q = \{A \subset \Ga : \{\al \in \Ga : A\al^{-1} \in p\} \in q\}.
$$
This extension of the product on $\Ga$ to a product on $\beta\Ga$
makes $\beta\Ga$ a {\em left topological semigroup}, i.e. one in which the maps
$L_q : \beta\Ga \to \beta\Ga$, defined by 
$p \mapsto q \cdot p = m_L(p,p)$, are continuous.

The {\em remainder space} of $\Ga$ is the compact space 
$X := \beta\Ga^* = \beta\Ga \setminus \Ga$.  Clearly
$X$ is a subsemigroup of $\beta\Ga$ with respect to both
right and left multiplications.
We let $Z_R : = \cls{X^2} =\cls\{pq: p,q \in X\}$ and
$Z_L : = \cls{X^{\cdot 2}} =\cls\{p \cdot q: p,q \in X\}$.
We also set $Z = Z_R \cup Z_L$.

\br

As the algebra $C_0(\Ga)$, comprising the functions on $\Ga$
which vanish at infinity, is contained in the algebra $WAP(\Ga)$
we deduce that $WAP(\Ga)$ distinguishes points in $\Ga$ and that
consequently the natural compactification map of $\Ga$ into $\Ga^{WAP}$ 
is an isomorphism.
We will therefore consider $\Ga$ as a dense discrete subset of both 
$\beta{\Ga}$ and $\Ga^{WAP}$.

A dynamical system $(X,\Ga)$ is called {\em weakly almost periodic}
(WAP) if for every $F \in C(X)$, its orbit $\{F \circ \ga : \ga \in \Ga\}$ forms
a weakly precompact subset of the Banach space $C(X)$. A theorem
of Ellis and Nerurkar which is based on well known results of 
Grothendiek asserts that a system $(X,\Ga)$ is WAP iff its enveloping
semigroup $E(X)$ consists of continuous maps, iff $E(X)$ is a
{\em semitopological semigroup} (that is, one in which both right and left multiplications
are continuous). It then follows that the dynamical system
$\Ga^{WAP}$ is the universal WAP point transitive dynamical system.
Moreover, $\Ga^{WAP}$ is isomorphic to its own enveloping 
semigroup and is therefore also the maximal semitopological semigroup compactification of $\Ga$.

Let $\pi : \beta\Ga \to \Ga^{WAP}$ denote the canonical homomorphism of the corresponding dynamical systems. With our identifications of $\Ga$ 
as a subset of both $\beta\Ga$ and $\Ga^{WAP}$ we have
$\pi(\ga) = \ga$ for every $\ga \in \Ga$. We set 
$Y : = \Ga^{WAP} \setminus \Ga$. 
As a direct consequence of the discussion above we see that for every 
$p, q \in \beta\Ga$ we have $\pi(pq) = \pi(p)\pi(q)$ 
and $\pi(p \cdot q)=\pi(p)\pi(q)$. Consequently $\pi(Z_R \cup Z_L) = \cls Y^2$. 
A result of our analysis shows that the restricted map
$$
\pi : \beta\Ga \setminus Z \to \Ga^{WAP} \setminus \cls Y^2
$$ 
is a homeomorphism (Theorem \ref{pi} below).
This extends results of Ruppert and Hindman and Strauss 
(see \cite[Theorem 21.22]{H-S}).

\br

\section{Translation-finite sets}

\begin{defns}\label{defn}
\begin{enumerate}
\item
Let $Z_R = \cls{X^2} \subset X$. Set
$$
\Ical_R = \{A \subset \Ga : \cls{A} \cap Z_R = \emptyset\}.
$$
Set 
$\Fcal_R = \{B \subset \Ga : B^c \in \Ical\}
=\{B \subset \Ga : \cls{B} \supset Z_R\}$.
Clearly $\Ical_R$ is an ideal and $\Fcal_R$ is a filter.
\item
Set $Z_L = \cls{X^{\cdot 2}} \subset X$, where
$X^{\cdot 2}= \{p \cdot q : p,q \in X\}$.
The ideal $\Ical_L$ and the filter $\Fcal_L$ are then defined
as above with $Z_L$ replacing $Z_R$.
\item
Let $Z = Z_R \cup Z_L \subset  X$. Set
$$
\Ical = \{A \subset \Ga : \cls{A} \cap Z = \emptyset\} = \Ical_R \cap \Ical_L.
$$
Set 
$\Fcal = \{B \subset \Ga : B^c \in \Ical\}=\{B \subset \Ga : \cls{B} \supset Z\}$.
Clearly then $\Ical$ is an ideal and $\Fcal = \Fcal_R \cap \Fcal_L$ is a filter.
\item
A subset $A \subset \Ga$ is called {\em right translation-finite} 
(RTF for short) if for every infinite
$D \subset \Ga$ there is a finite $F \subset D$ such that 
$\cap_{\del \in F}A\del^{-1}$ is finite.
We denote by $\Ical_{RTF}$ be the collection of RTF subsets of $\Ga$.
We say that a subset $B \subset \Ga$ is {\em co-right-translation-finite} (CRTF)
if $B^c = \Ga \setminus B$ is RTF and denote the collection of CRTF sets
by $\Fcal_{RTF}$.
Thus a subset $B \subset \Ga$ is CRTF if for every infinite subset
$D \subset \Ga$ there is a finite subset $F \subset D$ such that
$\cup_{\del \in F} B\del^{-1}$ is co-finite in $\Ga$.
These notions have obvious left analogues,  LTF subsets of $\Ga$,
$\Ical_{LTF}$, etc.
Following Ruppert we say that elements of 
$\Ical_{TF} := \Ical_{LTF} \cap \Ical_{RTF}$
are {\em translation-finite sets} (TF).
\item
We let $\Ical_W$ be the collection of sets $A \subset \Ga$ such that 
$\cls{A}$ is an open subset of $\Ga^{WAP}$ with 
$\cls{A} \cap \cls{Y^2} = \emptyset$. Then
$\Fcal_{W} = \{A^c : A \in \Ical_W\}$.
\item
We say that $A \subset \Ga$ is a {\em W-interpolation set}
if $\cls{A} \subset \Ga^{WAP}$ is an open subset of $\Ga^{WAP}$
which is homeomorphic to $\beta{A}$.
We let $\Ical_{IW}$ denote the collection of W-interpolation sets,
and let $\Fcal_{IW} = \{A^c : A \in \Ical_{IW}\}$.
\end{enumerate}
\end{defns}

Recall the following theorems of Ruppert (Theorem 7 and 
Proposition 13 in \cite{Ru}).

\begin{thm}\label{Ru}
\begin{enumerate}
\item
$\Ical_{TF}$ is an ideal and
$\Ical_{TF} = \Ical_{W} = \Ical_{IW}$.
\item
Every infinite subset of $\Ga$ contains an infinite TF subset.
\end{enumerate}
\end{thm}

\br \br 

Ruppert's main tools in analyzing the TF property were 
the universal WAP compactification of $\Ga$ and Grothendieck's
double limit characterization of WAP functions.
Our approach is through the \v{C}ech-Stone compactification of $\Ga$ and the
combinatorial definition of the product of ultrafilters. 

\begin{thm}\label{Fcal}
\begin{enumerate}
\item
$\Fcal_R=\Fcal_{RTF}$,
in particular $\Fcal_{RTF}$ is a filter.
\item
$\Fcal_L=\Fcal_{LTF}$,
in particular $\Fcal_{LTF}$ is a filter.
\item
$\Fcal=\Fcal_{TF} = \Fcal_{RTF} \cap \Fcal_{LTF}$, hence 
\begin{equation*}\label{RL}
\Ical_{TF} = \Ical_{LTF} \cap \Ical_{RTF} = \Ical_L \cap \Ical_R = \Ical  
= \Ical_{W} = \Ical_{IW}
\end{equation*}
\end{enumerate}
\end{thm}

\begin{proof}
We prove the two inclusions of claim (1) below. The claim (2) then holds by symmetry and claim (3) is obtained by taking the appropriate intersections
and applying Ruppert's theorem. 

{\bf{Side 1}:} \ 
We first show that $\Fcal_{RTF} \subset \Fcal_R$.
Consider $B \in \Fcal_{RTF}$ and suppose $A \subset \Ga$ has the property that
there are $p,q \in X$ with $A  \in pq$; i.e. $Ap^{\leftarrow}:
=\{\ga \in \Ga : A\ga^{-1} \in p\} \in q$. Then $|Ap^{\leftarrow}|=\infty$ and by assumption
there is a finite subset $F \subset Ap^{\leftarrow}$ such that
$\cup_{\del \in F} B \del^{-1}$ is cofinite in $\Ga$. As $p$ is an ultrafilter 
this implies that for some $\del \in F$ we have $B \del^{-1} \in p$. 
Now, as both $B \del^{-1}$ and $A\del^{-1}$ are in $p$ so is $(A\cap B)\del^{-1}$.
In particular we conclude that $A \cap B \not=\emptyset$.
This discussion shows that for any two ultrafilters $p,q$ in $X$
their product $pq$ is in $\cls{B}$; hence $\cls{B} \supset Z_R$,
i.e. $B \in \Fcal_R$.

{\bf{Side 2}:} \ 
Next we show that $\Fcal_R \subset \Fcal_{RTF}$.
Suppose then that $A \subset \Ga$ is not in $\Fcal_{RTF}$; i.e. there is
an infinite $D \subset \Ga$ such that for every finite $F \subset D$ we have
$| (AF^{-1})^c | = |\Ga \setminus \cup_{\del \in F} A \del^{-1}| =\infty$.
Clearly then the collection of sets of the form $(AF^{-1})^c$,
with $F \subset D$ finite, is a filter, say $\Lcal$, on $\Ga$. Choose some
ultrafilter $p \supset \Lcal$. Now choose an ultrafilter $q$ with $D \in q$.
We will show that $A \not \in pq$, whence $A \not \in \Fcal_R$, as required.

Assuming $A \in pq$ we have $Ap^{\leftarrow} = \{\ga \in \Ga: A\ga^{-1} \in p\}
\in q$. However if $\del \in D$ then 
$(A  \del^{-1})^c \in \Lcal$, hence $(A  \del^{-1})^c \in p$,
hence $A  \del^{-1} \not \in p$, hence
$D^c \supset Ap^{\leftarrow} \in q$, hence $D^c \in q$. 
This is a contradiction and we conclude that indeed $A \not \in pq$.
\end{proof}

\begin{thm}\label{pi}
\begin{enumerate}
\item
We have $\pi^{-1}(\cls{Y^2}) = Z = Z_R \cup Z_L$, hence
$\pi^{-1}(Y \setminus \cls{Y^2}) = X \setminus Z$.
\item
The restriction of $\pi$ to the open dense subset 
$X \setminus Z$ of $X$
is a homeomorphism from $X \setminus  Z$ onto 
$Y \setminus \cls{Y^2}$.
\end{enumerate}
\end{thm}

\begin{proof} 
{\bf{Step 1}:} \ 
Given $y \in U \subset (\Ga^{WAP} \setminus \cls{Y^2})$, where $y \in Y$ and 
$U$ is an open subset of $\Ga^{WAP}$,
let $V$ be an open subset of $\Ga^{WAP}$
such that $y \in V \subset \cls{V} \subset U$.
The set $\tilde{V}=\pi^{-1}(V)$ is an open subset of $\beta\Ga$ such that
$\cls{\tilde{V}} \cap Z = \emptyset$ 
(since $\pi$ is a homomorphism of semigroups we have 
$\pi(X^2) = \pi(X^{\cdot 2}) = Y^2$, for both right and left semigroup structures on $\beta\Ga$). 
Let $A = \Ga \cap \tilde{V}$, 
then $\cls_{\beta\Ga}{A} = \cls{\tilde{V}}$ and therefore $A \in \Ical$. 
By Theorem \ref{Fcal}.3 we have $A \in \Ical_{TF}$
and then, by Theorem \ref{Ru}, $A \in \Ical_{W}$. We conclude that 
$\cls_{\Ga^{WAP}}{A}$ is a clopen neighborhood of $y$ which is contained in 
$U$. Thus we have shown that the collection of sets of the form $\cls{A}$ 
with $A \in \Ical_{TF}$, is a basis for the topology on 
$\Ga^{WAP} \setminus \cls{Y^2}$. 

{\bf{Step 2}:} \ 
If $A$ is any set in $\Ical_{TF}$ then again by Theorem \ref{Ru},
$A \in \Ical_W = \Ical_{IW}$ and we conclude that $\cls{A}$ is a
clopen subset of $\Ga^{WAP}$ which is homeomorphic to $\beta A$.
By the universality of $\beta{A}$ it follows that 
$\pi : \cls_{\beta\Ga}{A} \to \cls_{\Ga^{WAP}}{A}$ is a homeomorphism.

{\bf{Step 3}:} \ 
Again if $A$ is any set in $\Ical_{TF}$ then, by Theorem \ref{Ru},
$A \in \Ical_W$ and we conclude that $\cls{A}$ is a
clopen subset of $\Ga^{WAP}$. We claim that $\pi^{-1}(\cls_{\Ga^{WAP}}{A}) =  \cls_{\beta\Ga}{A}$. Clearly $\cls_{\beta\Ga}{A}
\subset \pi^{-1}(\cls_{\Ga^{WAP}}{A})$.
Conversely, if $p \in \beta\Ga$ with 
$\pi(p) = y \in \cls_{\Ga^{WAP}}{A}$, let $p = \lim \ga_\nu$ for a net
$\ga_\nu \in \Ga$. Then 
$y = \pi(p) = \lim \pi(\ga_\nu) = \lim \ga_nu$ and, as by assumption
the set $\cls_{\Ga^{WAP}}{A}$ is a clopen subset of $\Ga^{WAP}$,
it follows that eventually $\ga_\nu \in A$. Thus we have $p 
\in \cls_{\beta\Ga}{A}$ as claimed.

{\bf{Step 4}:} \  
By Proposition 13 of \cite{Ru} (Theorem \ref{Ru}.2), every infinite subset 
$B \subset \Ga$
contains an infinite subset $A \subset B$ with $A \in \Ical_{TF}$.
In view of step 1 above this shows that the set 
$Y \setminus \cls{Y^2}$ is a dense open subset of $Y$.

{\bf{Step 5}:} \ 
Summing up we have shown that (i) the collection of clopen sets 
$\{\cls_{\Ga^{WAP}}{A} : A \in \Ical_{TF}\}$
forms a basis for the topology on $\Ga^{WAP} \setminus \cls{Y^2}$, (ii)  for each $A \in \Ical_{TF}$,
$\pi^{-1}(\cls_{\Ga^{WAP}}{A}) =  \cls_{\beta\Ga}{A}$ and moreover (iii)
$\pi : \cls_{\beta\Ga}{A} \to \cls_{\Ga^{WAP}}{A}$ is a homeomorphism.
These facts together with the fact that $Y \setminus \cls{Y^2}$ is a dense open subset of $Y$ prove the assertions of Theorem \ref{pi}.
\end{proof}

\br

\section{Divisible properties, IP and D sets}
In \cite{Gl-80} a collection $\Pcal$ of subsets of $\Ga$ is called
a {\em divisible property} if 
\begin{enumerate}
\item[(i)]
 $\emptyset \not\in \Pcal$
and $\Ga \in \Pcal$, 
\item[(ii)]
$\Pcal$ is hereditary upward (i.e. $A \in \Pcal$ and $B\supset A$
imply $B \in \Pcal$ and 
\item[(iii)]
 if $A\in \Pcal$ is a union 
$A =A_1 \cup A_2$ then at least one of the sets $A_1$ and 
$A_2$ is in $\Pcal$.
\end{enumerate}
A collection $\Pcal$ is {\em divisible} iff the collection 
$\Ical = \{A \subset \Ga: A \not\in \Pcal\}$ is an ideal iff 
the {\em dual} collection 
$\Fcal = \Pcal^* =\{A \subset \Ga : A \cap B \ne\emptyset, \ \forall B \in \Pcal\}$ 
is a filter.
When $\Fcal$ is a filter of subsets of $\Ga$ 
the compact (nonempty)
subset $K = \bigcap \{\cls{A} : A \in \Fcal\} \subset \beta\Ga$ is called 
the {\em kernel} of $\Fcal$. Conversely, any compact subset $K \subset
\beta\Ga$ defines a filter  
$$
\Fcal = \{A \subset \Ga : \cls{A} \supset K\}.
$$
The correspondence $\Fcal \leftrightarrow K$ is one to one and we note
that
$$
\Ical = \{A \subset \Ga : \cls{A} \cap K = \emptyset\} \quad{\text {\rm and}} \quad
\Pcal = \{A \subset \Ga : \cls{A} \cap K \not=\emptyset\},
$$
are the corresponding ideal and divisible properties respectively.

\br

Expressed explicitly the divisible property which corresponds
to the ideal of RTF-sets is the following one:
a subset $A \subset \Ga$ is not right translation finite, {\em an NRTF-set}, if there exists an infinite subset $D \subset \Ga$ such that for every finite subset 
$F \subset D$
the corresponding intersection $\bigcap_{\del \in F} A\del^{-1}$ is infinite.
NLTF-sets are defined similarly and a set $A$ is NTF if 
if there exists an infinite subset $D \subset \Ga$ such that for every finite subset 
$F \subset D$ at least one of the two corresponding intersections 
$\bigcap_{\del \in F} A\del^{-1}$ and $\bigcap_{\del \in F} \del^{-1}A$ is infinite.
In this terminology Theorem \ref{Fcal} is stated as follows:

\begin{thm}
The properties NRTF, NLTF and NTF are divisible
with corresponding kernels $Z_R, Z_L$ and $Z$ respectively.
\end{thm}

Note however that the ideal $\Ical_{W}$ is not what we call in 
\cite{Gl-80} the collection of interpolation sets of the algebra
$WAP(\Ga)$, as in Definition \ref{defn}.6 we postulate that $A \in \Ical_W$ 
when it is a $WAP(\Ga)$ interpolation set which additionally satisfies the requirement that $\ch_D \in WAP(\Ga)$. In \cite{Gl-80} (Corollary 5.3.2)
we have shown that the collection $\Jcal$ of WAP-interpolation sets 
has the property that if $\Ga = \bigcup_{i=1}^n A_i$ then at least
one of the sets $A_i$ is not in $\Jcal$. 
Let $\Ga_{dis}^{WAP}$ denote the {\em universal totally disconnected semitopological compactification} of $\Ga$. It is obtained as the
quotient $\Ga^{WAP}/{\sim}$ of $\Ga^{WAP}$ by the equivalence relation: 
$x \sim y \iff $ $x$ and $y$ lie in the same connected component. 
Let $WAP_{dis}(\Ga)$ denote the corresponding $C^*$-algebra.

\begin{prob}
(a) Is the collection of $WAP(\Ga)$-interpolation sets an ideal ?

(b) Is the collection of $WAP_{dis}(\Ga)$-interpolation sets an ideal ?
\end{prob}

\br

For simplicity let us assume next that $\Ga$ is abelian.
We will denote the group operation by $+$ but keep the notation
$(p,q) \mapsto pq$ for the semigroup operation on $\beta\Ga$.
Recall that a subset $A$ of $\Ga$ is a {\em D-set} if there is an infinite
sequence $\{\ga_i\}_{i=1}^\infty  \subset \Ga$ such that for every $i \ne j$
at least one of the elements $\ga_i - \ga_j$ or $\ga_j - \ga_i$ is in $A$.
The subset $A$ is called an {\em IP-set} if there is an infinite sequence 
$\{\ga_i\}_{i=1}^\infty  \subset \Ga$ such that for every
finite sequence $i_1 < i_2 < \cdots < i_n$ the element
$\ga_{i_1} + \ga_{i_2} + \cdots + \ga_{i_n}$ is in $A$.
It is well known that Hindman's theorem is equivalent to the fact that 
the collection of IP-sets is a divisible property with the set 
$K = \cls \{ v \in X : v^2 =v\}$ (the closure of the set of idempotents in $X$)
as its kernel. Obviously $K \subset Z$. 
It is easy to see that every IP-set is also a D-set.

The filter which corresponds
to the IP-sets is the collection of IP$^*$-sets:
$$
\{A \subset X : \cls{A} \supset K\} =
\{A \subset \Ga : 
A \cap B \ne\emptyset, \ \forall\ {\text{\rm{IP-set}}}\ B \}.
$$
Similarly the filter which corresponds
to the D-sets is the collection of D$^*$-sets:
$$
\{A \subset X : \cls{A} \supset K\} =\{A \subset \Ga : 
A \cap B \ne\emptyset, \ \forall\ {\text{\rm{D-set}}}\ B \}.
$$

The fact that the collection of D-sets is a divisible property is
equivalent to Ramsey's theorem and in \cite{Gl-80} we have identified the 
kernel of this divisible property as the following closed subset $L\subset X$. 
Define the set $V\subset X$ as follows:
$p \in X$ is in $V$ iff there is an element $q \in X$ and a net
$\ga_\al$ in $\Ga$ such that $\lim \ga_\al = q$ and 
$p = \lim \ga_\al^{-1}q$. Now put $L = \cls{V}$.

It is easy to see that $V \subset X^2$, whence $L \subset Z$. Thus
the identifications of the kernels $K$ and $L$, together with Theorem
\ref{Fcal}, immediately lead to the following corollary.

\begin{cor}
Every CTF-set (i.e. the complement of a TF-set) is a D$^*$-set and a 
fortiori an IP$^*$-set.
\end{cor}

\br

\bibliographystyle{amsplain}

\end{document}